\documentclass[a4paper,11pt]{article}
\usepackage{iwona,amsmath,amsthm,amsfonts,amssymb,enumerate,graphicx,float,wrapfig,calc,microtype,thmtools,underscore,mathtools}
\usepackage[usenames,dvipsnames,svgnames,table]{xcolor}
\usepackage[unicode=true]{hyperref}
\hypersetup{
colorlinks,
linkcolor={black},
citecolor={black},
urlcolor={blue!60!black},
pdftitle={Queue Layouts of Graphs with Bounded Degree and Bounded Genus},
pdfauthor={Vida Dujmovi{\'c}, Pat Morin, David~R.~Wood}} 
\usepackage[noabbrev,capitalise]{cleveref}

\usepackage[longnamesfirst,numbers,sort&compress]{natbib}
\makeatletter
\def\NAT@spacechar{~}
\makeatother

\usepackage[lmargin=29mm,rmargin=29mm,tmargin=29mm,bmargin=29mm]{geometry}
\renewcommand{\baselinestretch}{1.1}
\allowdisplaybreaks

\DeclarePairedDelimiter{\floor}{\lfloor}{\rfloor}

\DeclareMathOperator{\qn}{qn}

\renewcommand{\geq}{\geqslant}
\renewcommand{\leq}{\leqslant}

\DeclareMathOperator{\dist}{dist}

\renewcommand{\thefootnote}{\fnsymbol{footnote}}	
\allowdisplaybreaks

\newcommand{\arXiv}[1]{arXiv:\,\href{https://arxiv.org/abs/#1}{#1}}
\newcommand{\msn}[1]{MR:\,\href{https://www.ams.org/mathscinet-getitem?mr=MR#1}{#1}}

\newcommand{\doi}[1]{doi:\,\href{https://dx.doi.org/#1}{#1}}

\theoremstyle{plain}
\newtheorem{theorem}{Theorem}
\newtheorem{lemma}[theorem]{Lemma}
\newtheorem{corollary}[theorem]{Corollary}

\theoremstyle{definition}

\begin{document}

\author{
Vida Dujmovi{\'c}\,\footnotemark[3] 
\qquad Pat Morin\,\footnotemark[1] 
\qquad David~R.~Wood\,\footnotemark[5]}

\footnotetext[3]{School of Computer Science and Electrical Engineering, University of Ottawa, Ottawa, Canada (\texttt{vida.dujmovic@uottawa.ca}). Research  supported by NSERC and the Ontario Ministry of Research and Innovation.}

\footnotetext[1]{School of Computer Science, Carleton University, Ottawa, Canada (\texttt{morin@scs.carleton.ca}). Research  supported by NSERC.}

\footnotetext[5]{School of Mathematical Sciences, Monash   University, Melbourne, Australia  (\texttt{david.wood@monash.edu}). Research supported by the Australian Research Council.}

\sloppy

\title{\boldmath\bf Queue Layouts of Graphs \\
with Bounded Degree and Bounded Genus}
\maketitle

% \thanks{\textbf{MSC Classification}: ???}

\begin{abstract} 
Motivated by the question of whether planar graphs have bounded queue-number, we prove that planar graphs with maximum degree $\Delta$ have queue-number $O(\Delta^{2})$, which improves upon the best previous bound of $O(\Delta^6)$. More generally, we prove that graphs with bounded degree and bounded Euler genus have bounded queue-number. In particular graphs with Euler genus $g$ and maximum degree $\Delta$ have queue-number $O(g+\Delta^{2})$. As a byproduct we prove that if planar graphs have bounded queue-number, then graphs of Euler genus $g$ have queue-number $O(g)$. 
 \end{abstract}

\renewcommand{\thefootnote}{\arabic{footnote}}

%%%%%%%%%%%%
\section{Introduction}

 \citet{BFGMMRU18} recently proved that planar graphs with bounded (maximum) degree have bounded queue-number. We improve their bound and more generally show that graphs with bounded degree and bounded genus have bounded queue-number. 

First we introduce queue layouts and give the background to the above results. For a graph $G$ and integer $k\geq 0$, a \textit{$k$-queue layout} of $G$ consists of a linear ordering $\preceq$ of $V(G)$ and a partition $E_1,E_2,\dots,E_k$ of $E(G)$, such that for $i\in\{1,\dots,k\}$, no two edges in $E_i$ are nested with respect to $\preceq$. Here edges $vw$ and $xy$ are \textit{nested} if $v\prec x\prec y \prec w$.  The \textit{queue-number} of a graph $G$, denoted by $\qn(G)$, is the minimum integer $k$ such that $G$ has a $k$-queue layout. These definitions were introduced by Heath~et~al.~\citep{HLR92,HR92} as a dual to stack layouts (also called book embeddings). In a stack layout, no two edges in $E_i$ cross with respect to $\preceq$. Here edges $vw$ and $xy$ \textit{cross} if $v\prec x\prec w \prec y$

\citet{HLR92} conjectured that every planar graph has bounded queue number. This conjecture has remained open despite much research on queue layouts \citep{Wiechert17,DM-GD03,DujWoo05,HLR92,HR92,Hasunuma-DAM,Pemmaraju-PhD,RM-COCOON95,DMW05,DujWoo04,DPW04,Miyauchi-IEICE07,DMW17}. \citet{DujWoo04} observed that every graph with $m$ edges has a $O(\sqrt{m})$-queue layout using a random vertex ordering. Thus every planar graph with $n$ vertices has queue-number $O(\sqrt{n})$. \citet{DFP13} proved the first breakthrough on this topic, by showing that every planar graph with $n$ vertices has queue-number $O(\log^2 n)$. \citet{Duj15} improved this bound to $O(\log n)$ with a simpler proof. 

\citet{DMW17} established (poly-)logarithmic bounds for more general classes of graphs.\footnote{The \textit{Euler genus} of a graph $G$ is the minimum integer $k$ such that $G$ embeds in the orientable surface with $k/2$ handles (and $k$ is even) or the non-orientable surface with $k$ cross-caps. Of course, a graph is planar if and only if it has Euler genus 0; see \citep{MoharThom} for more about graph embeddings in surfaces. A graph $H$ is a \textit{minor} of a graph $G$ if a graph isomorphic to $H$ can be obtained from a subgraph of $G$ by contracting edges.} For example, they proved that every graph with $n$ vertices and Euler genus $g$ has queue-number $O(g+\log n)$, and that every graph with $n$ vertices excluding a fixed minor has queue-number $\log^{O(1)} n$. 

Recently, \citet{BFGMMRU18} proved a second breakthrough result, by showing that planar graphs with bounded degree have bounded queue-number. 

\begin{theorem}[\citep{BFGMMRU18}]
\label{Bekos}
Every planar graph with maximum degree $\Delta$ has queue-number at most $32(2\Delta-1)^6-1$.
\end{theorem}

%In particular, every planar graph with maximum degree $\Delta$ has queue-number $O(\Delta^6)$. 
Note that bounded degree alone is not enough to ensure bounded queue-number. In particular, \citet{Wood-QueueDegree} proved that for every integer $\Delta\geq 3$ and all sufficiently large $n$, there are graphs with $n$ vertices, maximum degree $\Delta$, and queue-number $\Omega(\sqrt{\Delta} n^{1/2-1/\Delta})$.

The first contribution of this paper is to improve the bound of \citet{BFGMMRU18} from $O(\Delta^6)$ to $O(\Delta^{2})$. 

\begin{theorem}
\label{Planar}
Every planar graph with maximum degree $\Delta$ has queue-number at most $12\Delta^2+16\Delta+3$. 
\end{theorem}

We extend this result by showing that graphs with bounded Euler genus and bounded degree have bounded queue-number. 

\begin{theorem}
\label{Genus}
Every graph with Euler genus $g$ and maximum degree $\Delta$ has queue-number at most  
$4g+36\Delta^2+48\Delta+9$. 
\end{theorem}

We remark that using well-known constructions \citep{DMW05,DPW04,SubQuad},  \cref{Genus} implies that graphs with bounded Euler genus and bounded degree have bounded track-number, which in turn can be used to prove linear volume bounds for three-dimensional straight-line grid drawings of the same class of graphs. These results can also be extended for graphs with bounded degree that can be drawn in a surface of bounded Euler genus with a bounded number of crossings per edge (using \citep[Theorem~6]{DujWoo05}). We omit all these details.

%\comment{\citet{DMW05} proved that every graph $G$ with acyclic chromatic number at most $ c$ and queue-number at most $k$ has track-number at most $c (2k)^{c-1}$. \citet{Borodin79} proved that planar graphs have acyclic chromatic number at most 5, and thus have track-number $O(\Delta^{8})$. Include this?}

The proof of \cref{Genus} uses \cref{Planar} as a `black box'. Starting with a graph $G$ of bounded Euler genus and bounded degree, we construct a planar  subgraph $G'$ of $G$. We then apply \cref{Planar} to obtain a queue layout of $G'$, from which we construct a queue layout of $G$. This approach suggests a direct connection between the queue-number of graphs with bounded Euler genus and planar graphs, regardless of degree considerations. The following theorem establishes this connection. A class of graphs is \emph{hereditary} if it is closed under taking induced subgraphs. 

\begin{theorem}
\label{general}
Let $\mathcal{G}$ be a hereditary class of graphs, such that every planar graph in $\mathcal{G}$ has queue-number at most $k$. Then every graph in $\mathcal{G}$ with Euler genus $g$ has queue-number at most $3k+4g$.
\end{theorem}

\cref{Genus} is an immediate corollary of \cref{Planar,general}, where $\mathcal{G}$ is the class of graphs with maximum degree at most $\Delta$.  \cref{general}, where $\mathcal{G}$ is the class of all graphs, implies the following result of interest:

\begin{corollary}
If every planar graph has queue-number at most $k$, then every graph with Euler genus $g$ has queue-number at most $3k+4g$.
\end{corollary}

For a graph $G$ and a set $A\subseteq V(G)$, let $G[A]$ be the subgraph of $G$ induced by $A$, which has vertex set $A$ and edge set $\{vw\in E(G): v,w\in A\}$. For disjoint sets $A,B\subseteq V(G)$, let $G[A,B]$ be the bipartite graph with bipartition $\{A,B\}$ and edge set $\{vw\in E(G): v\in A, w\in B\}$. 

%%%%%%%%%%%%%
\section{Planar Graphs of Bounded Degree}

This section proves \cref{Planar}. The proof is inspired by the proof of \cref{Bekos} by \citet{BFGMMRU18}. Here is high-level overview of their proof for a planar graph $G$ with maximum degree $\Delta$. First, \citet{BFGMMRU18} construct a particular planar graph $G_1$ obtained from $G$ by subdividing each edge at most three times. Then they construct a planar graph $G_2$ from $G_1$ by replacing certain edges by pairs of trees and a perfect matching between their leaves. $G_2$ is called a `$\Delta$-matched' graph. The heart of the proof of \citet{BFGMMRU18} is to construct a $O(\Delta)$-queue layout of any $\Delta$-matched graph, and thus of $G_2$. They then observe that the queue layout of $G_2$ also gives a $O(\Delta)$-queue layout of $G_1$. Finally, they use a generic lemma of \citet{DujWoo04}, which says that if some $(\leq c)$-subdivision of a graph has a $k$-queue layout, then the original graph has a $O(k^{2c})$-queue layout.  \citet{BFGMMRU18} apply this result with $k=O(\Delta)$ and $c=3$, to obtain a $O(\Delta^6)$-queue layout of $G$. 

It should be mentioned that there is a straightforward way to improve this $O(\Delta^6)$ bound. \cref{NewUnsubdivide} in Appendix~1 shows that if some $(\leq c)$-subdivision of a graph has a $k$-queue layout for some fixed $c$, then the original graph has a $O(k^{c+1})$-queue layout. Moreover, in the proof of \citet{BFGMMRU18}, for every edge $e$ of $G$ that is subdivided three times, one of the edges in the subdivision of $e$ is assigned to a single queue ($\mathcal{Q}_0$ in their notation). This observation, in conjunction with the proof of \cref{NewUnsubdivide}, leads to a $O(\Delta^3)$-queue layout of $G$.

Our proof of \cref{Planar} initially follows a similar strategy. Starting with a planar graph $G$ with maximum degree $\Delta$, we consider the $(\leq 3)$-subdivision $G_1$ of $G$  constructed by \citet{BFGMMRU18}. Note that \citet{BFGMMRU18} explain in Section~3.3 of their paper that one can work directly with $G_1$ instead of the $\Delta$-matched graph $G_2$, and this is what we choose to do. The key properties of $G_1$ are summarised in the definition of `well-layered' below. We then construct a partition of $V(G_1)$ with several desirable properties (see \cref{WellLayered}). This partition is implicit in the proof of \citet{BFGMMRU18}---there is really nothing new in this part of our proof. 

The main point of difference between our proof and that of \citet{BFGMMRU18} is that we do not apply the generic `unsubdividing' lemma of \citet{DujWoo04}. Instead we refine the partition of $V(G_1)$ to obtain a similar partition of $V(G)$ (see \cref{Structure}). From this partition one can determine a $O(\Delta^2)$-queue layout of $G$. Note that in this $O(\Delta^2)$-queue layout, the vertex ordering is identical to that used by \citet{BFGMMRU18}, only the queue assignment is different. This fact shows the value in focusing on structural partitions rather than the final queue layout. %The final step of our argument uses the Lova\'sz Local Lemma to show that one can permute $O(\Delta)$ vertices within each part of our  partition to obtain an even stronger partition of $V(G)$ (see \cref{??}), from which we obtain a $O(\Delta)$-queue layout.

The following definitions are key concepts in our proofs (and that of several other papers on queue layouts \citep{DMW17,DMW05,DujWoo04,BFGMMRU18}). A \emph{layering} of a graph $G$ is a partition $(V_0,V_1,\dots,V_t)$ of $V(G)$ such that for every edge $vw\in E(G)$, if $v\in V_i$ and $w\in V_j$, then $|i-j| \leq 1$. If $r$ is a vertex in a connected graph $G$ and $V_i:=\{v\in V(G):\dist_G(r,v)=i\}$ for all $i\geq 0$, then $(V_0,V_1,\dots,V_t)$ is called a \textit{BFS layering} of $G$, where $t:=\max\{\dist_G(r,v):v\in V(G)\}$. Associated with a bfs layering is a \emph{bfs spanning tree} $T$ obtained by choosing, for each non-root vertex $v\in V_i$ with $i\geq 1$, a neighbour $w$ in $V_{i-1}$, and adding the edge $vw$ to $T$. Thus $\dist_T(r,v)=\dist_G(r,v)$ for each vertex $v$ of $G$. When the spanning tree $T$ is obvious from the context, we call edges in $T$ \emph{tree edges} and edges not in $T$ \emph{non-tree edges}. An edge $vw\in E(G)$ with $v,w\in V_i$ for some $i\geq 0$ is called a \emph{level} edge. An edge $vw\in E(G)$ with $v\in V_i$ and $w\in V_{i+1}$ for some $i\geq 0$ is called a \emph{binding} edge. Every tree edge is \emph{binding}. 

The following lemma of \citet{Pupyrev17} shows that every planar graph has a drawing that highlights particular aspects of a BFS layering, as illustrated in \cref{EdgeTypes}. 

\begin{lemma}[\citep{Pupyrev17}] 
\label{concentric}
For every connected planar graph $G$ and every vertex $r$ of $G$, if $T$ is the BFS tree and $(V_0,V_1,\dots,V_t)$ is the BFS layering of $G$ rooted at $r$, then there is a drawing of $G$ in $\mathbb{R}^2$ with the $r$ at the origin and on the outer-face, such that for $i\in\{1,2,\dots,t\}$, 
\begin{itemize}
\item the vertices in $V_i$ are drawn on a circle $C_i$ of radius $R_i$ centred at the origin, where $0< R_1< R_2<\dots<R_t$; 
\item each level edge $vw\in E(G)$ with $v,w\in V_i$ is drawn as an open curve between $v$ and $w$ strictly outside of $C_i$; and 
\item each binding edge $vw$ with $v\in V_i$ and $w\in V_{i+1}$ is drawn either:
\begin{itemize}
\item as an open curve from $v$ to $w$ strictly between $C_i$ and $C_{i+1}$ (called a \emph{direct} edge), or 
\item as an open curve starting at $v$ that crosses $C_{i+1}$ once at a point distinct from $w$, then stays outside of $C_{i+1}$, and ends at $w$ (called a \emph{hooked} edge).  
\end{itemize}
\item each tree edge $vw\in E(T)$ is direct and binding..
\end{itemize}
\end{lemma}

\begin{figure}[!h]
\includegraphics{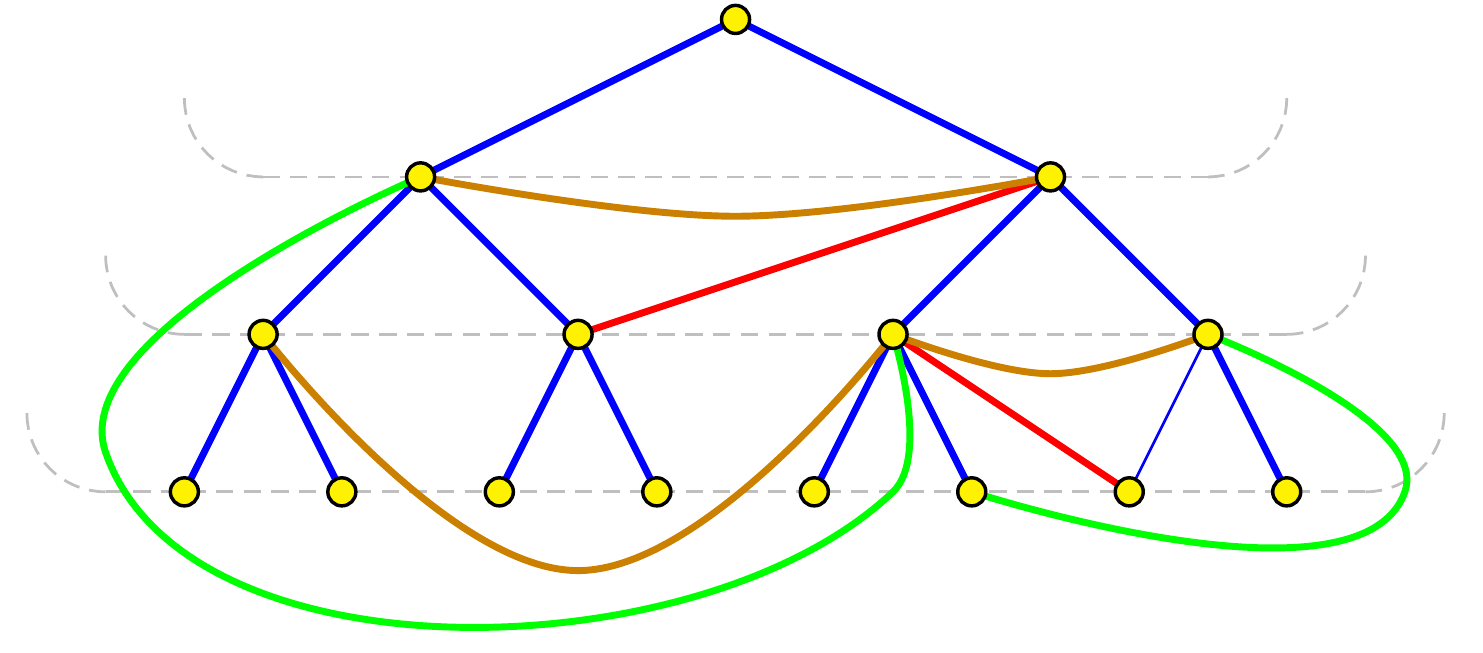}
\caption{Drawing of planar graph on concentric circles: tree edges are blue, hooked edges are green, level edges are brown, direct non-tree edges are red.\label{EdgeTypes}}
\end{figure}

%%%%%%%%%%%
\subsection{Well-Layered Planar Graphs}

A planar graph $G$ is \emph{well-layered} if there is a BFS spanning tree $T$ of $G$ rooted at a vertex $r$ such that every non-tree edge $vw\in E(G)\setminus E(T)$ is a level edge in the corresponding BFS layering, and both $v$ and $w$ are leaves in $T$ with degree 2 in $G$. This implies that the set of non-tree edges are a matching in $G$. 
 
\begin{lemma}
\label{WellLayered}
Let $G$ be well-layered planar graph with corresponding BFS spanning tree $T$ and BFS layering $(V_0,V_1,\dots,V_t)$ rooted at a vertex $r$. Assume that every vertex in $G$ has at most $\Delta$ children in $T$. Then for $i\in\{1,2,\dots,t\}$, there is a partition $\{V_{i,a}: a\geq 0\}$ of $V_i$, and an ordering $\overrightarrow{V_{i,a}}$ of each set $V_{i,a}$, such that:
\begin{enumerate}[(a)]
\item\label{WellLayered-NonTreeSpan} 
for each non-tree edge $vw\in E(G)\setminus E(T)$, both $v$ and $w$ are in $V_{i,a}$ for some $i,a\geq 0$,
\item\label{WellLayered-TreeSpan} 
for each tree edge $vw\in E(T)$, if $v\in V_{i,a}$ and $w\in V_{i+1,b}$ for some $a,b\geq 0$, then $b-\Delta a \in\{0,1,\dots,2\Delta-1\}$,
\item\label{WellLayered-InBag} 
for all $i,a\geq 0$, no two edges in $G[V_{i,a}]$ cross or nest with respect to the ordering $\overrightarrow{V_{i,a}}$, in particular, $\overrightarrow{V_{i,a}}$ defines a $1$-queue layout of $G[V_{i,a}]$, and
\item\label{WellLayered-InterBag} 
for all $i,a,b\geq 0$, the ordering $\overrightarrow{V_{i,a}}\overrightarrow{V_{i+1,b}}$ defines a 1-queue layout of $G[V_{i,a},V_{i+1,b}]$. 
\end{enumerate}
\end{lemma}

\begin{proof}
Apply \cref{concentric} to obtain a drawing of $G$ on concentric circles $C_1,C_2,\dots,C_t$. 
For each vertex $v\in V(G)$, let $\ell(v):=\dist_G(v,r)-t$. Thus $\ell(v)\in\{0,1,\dots,t\}$ for every vertex $v$ of $G$, and $\ell(r)=t$. For each vertex $v$ of $G$, let $T_v$ be the subtree of $T$ rooted at $v$. 

For each non-tree edge $vw\in E(G)\setminus E(T)$, let $D_{vw}$ be the cycle obtained from the $vw$-path in $T$ by adding the edge $vw$. Note that if $v,w\in V_i$ then the $vw$-path in $T$ is drawn within the interior of $C_i$ (since every tree edge is binding and direct) and $vw$ is drawn outside of $C_i$. 

Let $G^+$ be the multigraph with vertex set $V(G)$, where each tree edge $vw\in E(T)$ has multiplicity 1 in $G^+$, and each non-tree edge $vw\in E(G)\setminus E(T)$ has multiplicity $\Delta^{\ell(v)}$ (which equals $\Delta^{\ell(w)}$ since every non-tree edge is a level edge). Note that $T$ is a spanning tree of $G^+$. 

A key property of this construction is that for each vertex $v$ of $G$, the number of non-tree edges in $G^+$ with one endpoint in $T_v$ is at most $\Delta^{\ell(v)}$. We prove this claim by induction on $\ell(v)$. First note that if $v$ is a leaf in $T$, then $T_v$ is simply the vertex $v$, and $v$ is incident to at most one non-tree edge in $G$, and thus is incident to at most $\Delta^{\ell(v)}$ edges in $G^+$. So the claim holds for leaves. In particular, if $\ell(v)=0$ then $v$ is a leaf in $T$, and the claim holds. Now consider a non-leaf vertex $v$ with $\ell(v)\geq1$. Let $v_1,\dots,v_d$ be the children of $v$, where $v_1,\dots,v_p$ are not leaves of $T$, and $v_{p+1},\dots,v_d$ are leaves of $T$. Note that $\ell(v_i)=\ell(v)-1$ for $i\in\{1,\dots,d\}$. By induction, for each $i\in\{1,\dots,p\}$, the number of non-tree edges in $G^+$ with one endpoint in $T_{v_i}$ is at most $\Delta^{\ell(v_i)}=\Delta^{\ell(v)-1}$. For $i\in\{p+1,\dots,d\}$, we have already shown that the number of non-tree edges in $G^+$ incident to $v_i$ is at most $\Delta^{\ell(v_i)}=\Delta^{\ell(v)-1}$ edges. In total, there are at most $d\Delta^{\ell(v)-1}\leq\Delta^{\ell(v)}$ non-tree edges in $G^+$ incident to $T_v$, as claimed. 

Let $G^*$ be the dual of $G^+$. Let $T^*$ be the spanning subgraph of $G^*$ consisting of those edges of $G^*$ dual to edges of $E(G^+)\setminus E(T)$. It is well known (and easily follows from Euler's formula) that $T^*$ is a spanning tree of $G^*$. ($T^*$ is sometimes called a \emph{co-tree}; note that $T$ and $T^*$ can be simultaneously drawn without crossing each other.)\ Let $r^*$ be the vertex of $T^*$ dual to the outer-face of $G^+$. Consider $T^*$ to be rooted at $r^*$. 

For each face $f$ of $G^+$, let $d(f)$ be the distance in $T^*$ between $r^*$ and the vertex of $T^*$ dual to $f$. For each vertex $v$ of $G^+$, let $m(v)$ be the minimum of $d(f)$ taken over all faces $f$ of $G^+$ incident with the subtree of $T$ rooted at $v$, and let $$g(v):=\floor*{\frac{m(v)}{\Delta^{\ell(v)}}}.$$ For $i,a\geq 0$, let $$V_{i,a}:=\{v\in V_i:g(v)=a\}.$$ 
Let $\overrightarrow{V_{i,a}}$ be the ordering of $V_{i,a}$ along circle $C_i$, where the outer-face defines the start and end point. (In the language of \citet{BFGMMRU18}, $m(v)$ is analogous to the `matching-value' of $v$, and $g(v)$ is the `layer-group' of $v$.) 

This concludes the description of the partition $\{V_{i,a}:i,a\geq 0\}$ and the orderings $\overrightarrow{V_{i,a}}$. We now show these satisfy the claims of the lemma. 

We now prove (\ref{WellLayered-NonTreeSpan}). Consider a non-tree edge $vw\in E(G)\setminus E(T)$. By assumption, both $v$ and $w$ are in $V_{i}$ for some $i\geq 0$, and both $v$ and $w$ are leaves in $T$. Thus $\ell(v)=\ell(w)$, and $v$ is the only vertex in the subtree rooted at $v$, and similarly for $w$. Since $\deg_G(v)=\deg_G(w)=2$, the faces incident to $v$ are exactly the same faces incident to $w$. Thus $m(v)=m(w)$, implying $v$ and $w$ are in $V_{i,a}$ where $a=\floor{m(v)/\Delta^{\ell(v)}}$. This proves (\ref{WellLayered-NonTreeSpan}). 

%\item for each edge $vw\in E(G)$, if $v\in V_{i,a}$ and $w\in V_{i+1,b}$ for some $i\geq 0$, then $b-a\Delta \in [0,3???\Delta]$, 

We now prove (\ref{WellLayered-TreeSpan}). Consider a non-leaf vertex $v$ with $\ell(v)=\ell$. Then all the edges incident to $v$ are in $T$. Let $x$ and $y$ be two children of $v$ consecutive in the embedding of $G$. Observe that $m(y)-m(x)$ is maximised when all the non-tree edges incident to $T_x$ go `under' $T_y$. The number of such edges is at most $\Delta^{\ell(x)}=\Delta^{\ell-1}$. Thus $m(y) \leq m(x) +\Delta^{\ell-1}$. Since $v$ has at most $\Delta$ children, $m(y) \leq m(x) + \Delta^{\ell}$ for all children $x$ and $y$ of $v$. Every face incident with $T_v$ is incident to $T_x$ for some child $x$ of $v$. Thus $m(v)$ equals the minimum of $m(x)$ taken over all children $x$ of $v$. Hence $m(y)\leq m(v) +  \Delta^{\ell}$ for all children $y$ of $v$, implying
$$g(y) = \floor*{\frac{m(y)}{\Delta^{\ell-1}}} \leq \frac{m(y)}{\Delta^{\ell-1}} 
\leq \Delta \frac{m(v)}{\Delta^{\ell}} +  \Delta 
< \Delta \left( \floor*{\frac{m(v)}{\Delta^{\ell}}} +1 \right) +  \Delta 
= \Delta\, g(v) +  2\Delta.$$ 
Since $g(v)$ and $g(y)$ are integers, $g(y) \leq \Delta\, g(v) + 2\Delta -1$. Moreover, $m(v) \leq m(y)$, implying
$$\Delta\, g(v) = \Delta \floor*{\frac{m(v)}{\Delta^{\ell}}}
\leq \Delta \,\frac{m(v)}{\Delta^{\ell}} \leq \frac{m(y)}{\Delta^{\ell-1}} < 
\floor*{\frac{m(y)}{\Delta^{\ell-1}}}+1 = g(y)+1.$$
Since $g(v)$ and $g(y)$ are integers, $g(y) \geq \Delta\, g(v)$. 
Summarising, if $y$ is a child of $v$ in $T$ then 
$$g(y) - \Delta\, g(v) \in\{0,1,\dots,2\Delta-1\}.$$
This says that for each tree edge $vw\in E(T)$ where $v\in V_{i,a}$ and $w\in V_{i+1,b}$, we have $b-\Delta a \in\{0,1,\dots,2\Delta-1\}$, which proves (\ref{WellLayered-TreeSpan}). 

%\item for all $i,a\geq 0$, the ordering $\overrightarrow{V_{i,a}}$ defines a 1-queue layout of $G[V_{i,a}]$, and
%Recall that $V_i =V_{i,0},V_{i,1},\dots$  for each $i\geq 0$. 
%We now prove that $V_i$ defines a 1-queue layout of $G[V_i]$. 
We now prove (\ref{WellLayered-InBag}), which claims that no two edges in $G[V_{i,a}]$ cross or nest with respect to the ordering $\overrightarrow{V_{i,a}}$. 
Consider edges $vw,pq\in E(G)$ with $v,w,p,q\in V_{i,a}$. 
Let $\ell:=\ell(v)=\ell(w)$. 
Neither $vw$ nor $pq$ are tree edges. 
Suppose on the contrary that $vw$ and $pq$ cross with respect to $\overrightarrow{V_{i,a}}$. 
Without loss of generality, $v \prec p \prec w \prec q$ in $\overrightarrow{V_{i,a}}$. 
Then $v,p,w,q$ appear in this order on the circle $C_i$. 
Since $vw$ and $pq$ are drawn outside $C_i$, 
these edges cross in the drawing of $G$, which is a contradiction. 
Thus no two edges in $G[V_{i,a}]$ cross with respect to $\overrightarrow{V_{i,a}}$.
Now suppose that $vw$ and $pq$ nest with respect to $\overrightarrow{V_{i,a}}$.
Without loss of generality, $v \prec p \prec q \prec w$ in $\overrightarrow{V_{i,a}}$. 
Thus $v,p,q,w$ appear in this order on $C_i$. 
Hence both $T_p$ and $T_q$ are inside $D_{vw}$ and the outer-face of $G^+$ is outside $D_{vw}$. Since $vw$ is the only edge of $D_{vw}$ not in $T$, every path in $T^*$ from $r^*$ to a vertex dual to a face incident with $p$ or $q$ must include the edges of $T^*$ dual to $vw$. Let $f$ be the face of $G$ immediately below $vw$. Since $vw$ has multiplicity $\Delta^{\ell}$ in $G^+$, for every face $f'$ incident with $T_p$ or $T_q$, we have $d(f')\geq d(f)+\Delta^{\ell}$. Since $w$ is incident with $f$, we have $m(w)\leq d(f)$. 
Thus $m(p) \geq d(f) + \Delta^\ell \geq m(w) + \Delta^\ell$. Hence
$$\frac{ m(p) }{ \Delta^\ell } \geq \frac{ m(w) }{\Delta^\ell} + 1.$$
Therefore 
$$g(p) = \floor*{ \frac{ m(p) }{ \Delta^\ell } } > \floor*{ \frac{ m(w) }{\Delta^\ell} } = g(w),$$
which implies that $p$ and $w$ are not both in $V_{i,a}$. 
This contradiction shows that $\overrightarrow{V_{i,a}}$ defines a 1-queue layout of $G[V_{i,a}]$ for all $i,a\geq 0$. This proves (\ref{WellLayered-InBag}). 

We now prove (\ref{WellLayered-InterBag}). 
%\item for all $i,a,b\geq 0$, the ordering $\overrightarrow{V_{i,a}}\overrightarrow{V_{i+1,b}}$ defines a 1-queue layout of $G[V_{i,a},V_{i+1,b}]$. 
Suppose on the contrary that $v \prec  x$ in $\overrightarrow{V_{i,a}}$ and $y \prec  w$ in $\overrightarrow{V_{i+1,b}}$ for some edges $vw,xy\in E(G)$ for some $i,a,b\geq 0$. Thus $v$ is to the left of $x$ in $C_i$ and $y$ is to the left of $w$ in $C_{i+1}$. Since every non-tree edge is a level edge, both $vw$ and $xy$ are tree edges, which are drawn direct between $C_i$ and $C_{i+1}$. Thus $vw$ and $xy$ cross. 
This contradiction shows that no two edges of $G[V_{i,a},V_{j,b}]$ are nested in the ordering $\overrightarrow{V_{i,a}}\overrightarrow{V_{i+1,b}}$, which thus defines a 1-queue layout of $G[V_{i,a},V_{i+1,b}]$. This proves (\ref{WellLayered-InterBag}). 
\end{proof}

Note that \cref{WellLayered} implies that every well-layered graph has a $2\Delta$-queue layout, as proved by \citet{BFGMMRU18}. To see this, let $\overrightarrow{V_i}$ be the ordering $\overrightarrow{V_{i,0}}\overrightarrow{V_{i,1}}\dots$ of $V_i$. Then take the ordering $\overrightarrow{V_0}\overrightarrow{V_1}\overrightarrow{V_2}\ldots$ of $V(G)$. By \cref{WellLayered}(\ref{WellLayered-InBag}), every level edge can be assigned to a single queue $Q^*$. Assign each tree edge $vw\in E(T)$ where $v\in V_{i,a}$ and $w\in V_{i+1,b}$ to $Q_{b-\Delta a}$. By \cref{WellLayered}(\ref{WellLayered-TreeSpan}) this introduces $2\Delta$ queues. Suppose that tree edges $vw$ and $pq$ in $Q_j$ are nested for some $j\in\{0,1,\dots,2\Delta-1\}$, with $v \prec p \prec q \prec w$ in the ordering. Then $v\in V_{i,a}$, $p\in V_{i,b}$, $q\in V_{i+1,c}$ and $w\in V_{i+1,d}$ for some $i,a,b,c,d\geq 0$ with $j=d-\Delta a = c- \Delta b$. Thus $d-c=\Delta(a-b)$. Since $v \prec p \prec q \prec w$ in the ordering, $d-c\geq 0$ and $a-b\leq 0$. Thus $a=b$ and $c=d$, which contradicts \cref{WellLayered}(\ref{WellLayered-InterBag}). 

\subsection{General Planar Graphs}

We now extend \cref{WellLayered} for all planar graphs.

\begin{lemma}
\label{Structure}
Let $G$ be a planar graph with a BFS spanning tree $T$ and BFS layering $(V_0,V_1,\dots,V_t)$ rooted at a vertex $r$. Assume that every vertex in $G$ has degree at most $\Delta+1$ and has most $\Delta$ children in $T$. 
Then for $i\in\{1,2,\dots,t\}$, there is a partition $\{V_{i,a}: a\geq 0\}$ of $V_i$, and an ordering $\overrightarrow{V_{i,a}}$ of each set $V_{i,a}$, such that: 
\begin{enumerate}[(a)]
\item\label{Structure-LevelEdgeSpan}
 for each level edge $vw\in E(G)$, if $v\in V_{i,a}$ and $w\in V_{i,b}$ then $|a-b|\leq 1$; 
\item\label{Structure-TreeEdgeSpan} 
for each tree edge $vw\in E(T)$, if $v\in V_{i,a}$ and $w\in V_{i+1,b}$ then $b-a\Delta\in\{0,1,\dots,2\Delta-1\}$;
\item\label{Structure-NonTreeBinding} 
 for each non-tree binding edge $vw\in E(G)\setminus E(T)$, if $v\in V_{i,a}$ and $w\in V_{i+1,b}$ then $b-a\Delta\in\{-1,0,\dots,2\Delta\}$;
\item\label{Structure-AdjacentBags} 
% the ordering $\overrightarrow{V_{i,a}}\overrightarrow{V_{i,a+1}}$ defines a $2\Delta$-queue layout of  $G[V_{i,a},V_{i,a+1}]$ for all $i,a\geq 0$.
the ordering $\overrightarrow{V_{i,a}}\overrightarrow{V_{i,a+1}}$ defines a $1$-queue layout of  $G[V_{i,a},V_{i,a+1}]$ for all $i,a\geq 0$.
\item\label{Structure-InterBag}  
the ordering $\overrightarrow{V_{i,a}}\overrightarrow{V_{i+1,b}}$ defines a $(6\Delta+1)$-queue layout of  $G[V_{i,a},V_{i+1,b}]$ for all $i,a,b\geq 0$.
\item\label{Structure-InBag}
the ordering $\overrightarrow{V_{i,a}}$ defines a $2\Delta$-queue layout of $G[V_{i,a}]$ for all $i,a\geq 0$.
\end{enumerate}
\end{lemma}

\begin{proof}
Apply \cref{concentric} to obtain a drawing of $G$ on concentric circles $C_1,C_2,\dots,C_t$ rooted at $r$. 

Let $G'$ be obtained by subdividing edges of $G$ as follows, as illustrated in \cref{Subdivide}. Initialise $T':=T$ and $V'_i:=V_i$ for each $i\geq 0$. For each level edge $vw\in E(G)$ with $v,w\in V_i$ for some $i\geq 0$: 
\begin{itemize}
\item replace $vw$ by a path $vxyw$ in $G'$ (where $x$ and $y$ are new vertices); 
\item add the edges $vx$ and $wy$ to $T'$ (so $x$ and $y$ are leaves in $T'$ with degree 2 in $G'$); and 
\item add $x$ and $y$ to $V'_{i+1}$. 
\end{itemize}
For each non-tree binding edge $vw\in E(G)\setminus E(T)$ with $v\in V_i$ and $w\in V_{i+1}$ for some $i\geq 0$: 
\begin{itemize}
\item replace $vw$ by a path $vxyzw$ in $G'$ (where $x,y,z$ are new vertices); 
\item add the edges $vx$, $xy$ and $wz$ to $T'$ (so $y$ and $z$ are leaves in $T'$ with degree 2 in $G'$); and 
\item add $x$ to $V'_{i+1}$, and add $y$ and $z$ to $V'_{i+2}$. 
\end{itemize}
Observe that $T'$ is a bfs spanning tree of $G'$, and $G'$ is well-layered with respect to the layering $V'_0,V'_1,\dots,V'_t$. For $i\geq 0$, let $\{V'_{i,a}: a\geq 0\}$ be the partition of $V'_i$ from \cref{WellLayered} applied to $G'$. Let $V_{i,a}:=V'_{i,a}\cap V(G)$ for $i,a\geq 0$, where $\overrightarrow{V_{i,a}}$ inherits its order from $\overrightarrow{V'_{i,a}}$. 

\begin{figure}
\centering\includegraphics{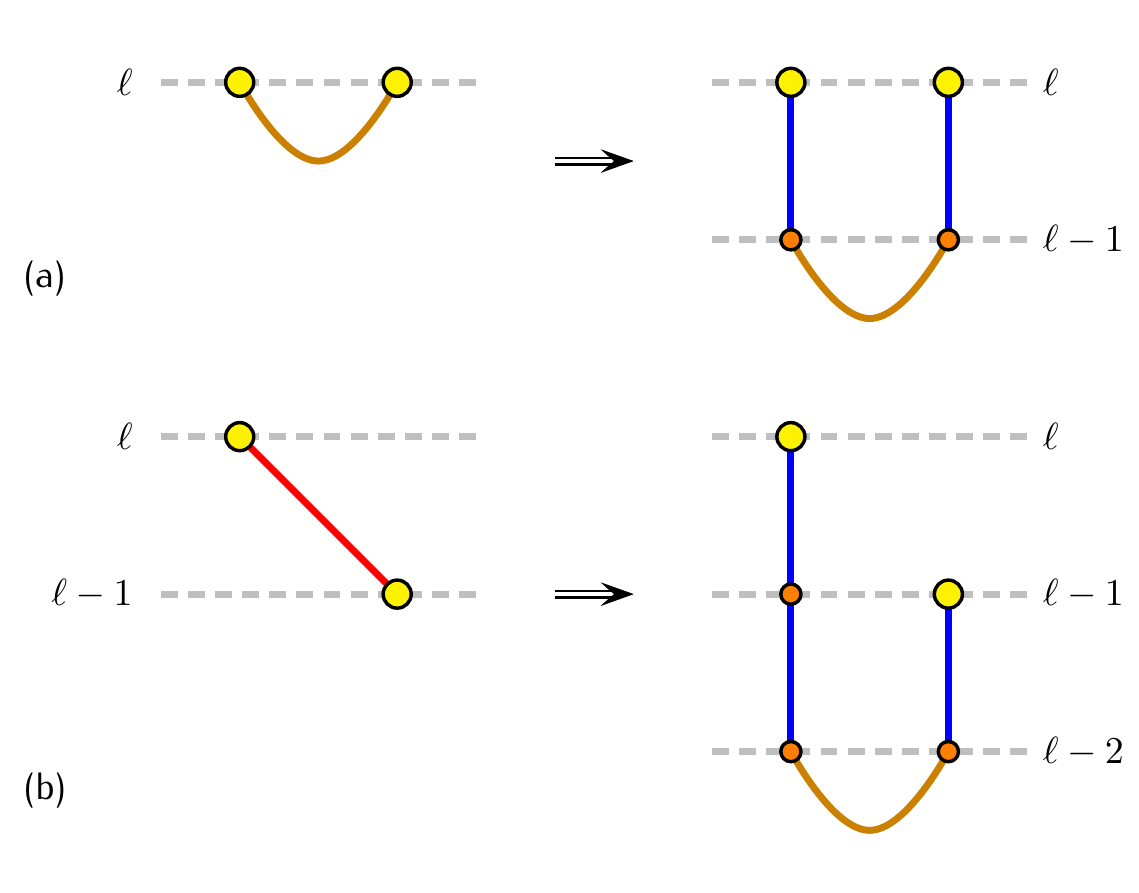}
\caption{\label{Subdivide} Creating the subdivision $G'$:  (a) level edge, (b) non-tree binding edge.}
\end{figure}

% for each level edge $vw\in E(G)$, if $v\in V_{i,a}$ and $w\in V_{i,b}$ then $|a-b|\leq 1$; 
We now prove (\ref{Structure-LevelEdgeSpan}). Consider a level edge $vw\in E(G)$ with $v\in V_{i,a}$ and $w\in V_{i,b}$. 
Let $vxyw$ be the corresponding path in $G'$. Then $xy$ is a level non-tree edge of $G'$ with $v,w\in V'_{i+1}$. 
By \cref{WellLayered}(\ref{WellLayered-NonTreeSpan}), both $x$ and $y$ are in $V'_{i+1,c}$ for some $c\geq 0$.
Since $vx$ and $wy$ are tree edges in $G'$
by \cref{WellLayered}(\ref{WellLayered-TreeSpan}), we have 
$c-\Delta a = \alpha$ and $c-\Delta b = \beta$ for some $\alpha,\beta\in\{0,1,\dots,2\Delta-1\}$.
Thus $c = \alpha+\Delta a= \beta+\Delta b$, implying 
$\Delta( a-b) =\beta-\alpha \leq 2\Delta-1$ and $ a-b \leq 1$. 
Similarly $b-a \leq 1$. Thus $|a-b|\leq 1$. This proves (\ref{Structure-LevelEdgeSpan}). 

Property (\ref{Structure-TreeEdgeSpan}) 
follows immediately from \cref{WellLayered}(\ref{WellLayered-TreeSpan}) 
since a tree edge $vw\in E(T)$ with $v\in V_{i,a}$ and $w\in V_{i+1,b}$
is a tree edge in $G'$ with $v\in V'_{i,a}$ and $w\in V'_{i+1,b}$, in which case  $b-\Delta a \in\{0,1,\dots,2\Delta-1\}$.

We now prove (\ref{Structure-NonTreeBinding}). 
Consider a binding non-tree edge $vw\in E(G)\setminus E(T)$ with $v\in V_{i,a}$ and $w\in V_{i+1,b}$.
Let $vxyzw$ be the corresponding path in $G'$.
Then $vx$, $xy$ and $wz$ are tree edges in $G'$, and $yz$ is a level edge in $G'$. 
Moreover,  $x\in V'_{i+1}$ and $y,z\in V'_{i+2}$. 
Then $x\in V'_{i+1,c}$ for some $c\geq 0$, 
and $y,z\in V'_{i+2,d}$ for some $d\geq 0$ by \cref{WellLayered}(\ref{WellLayered-NonTreeSpan}).
Since $vx$, $xy$ and $wz$ are tree edges, by \cref{WellLayered}(\ref{WellLayered-TreeSpan}), 
$c-\Delta a =\alpha$ and $d-\Delta c =\beta$ and $d-\Delta b =\gamma$ for some $\alpha,\beta,\gamma\in\{0,1,\dots,2\Delta-1\}$. 
Thus $d = \beta + \Delta c = \gamma + \Delta b$, implying
$\Delta (c-b) = \gamma-\beta \leq 2\Delta-1$ and $c-b \leq 1$. 
Similarly, 
$\Delta (b-c) = \beta-\gamma \leq 2\Delta-1$ and $b-c \leq 1$. 
Now
$b  \leq c +1 =  \alpha + a \Delta+1$, implying
$b - a \Delta  \leq \alpha +1 \leq 2\Delta$. 
Similarly, 
$\Delta (c-b) = \gamma-\beta \leq 2\Delta-1$ and $c-b \leq 1$. 
Thus
$b\geq c-1 = \Delta a + \alpha - 1 \geq \Delta a-1$, 
implying $b - \Delta a \geq -1$. 
In summary, $b - a\Delta \in\{-1,0,\dots,2\Delta\}$. 
This proves (\ref{Structure-NonTreeBinding}). 

We now prove (\ref{Structure-AdjacentBags}). 
Suppose on the contrary that edges $vw$ and $pq$ in $G[V_{i,a},V_{i,a+1}]$ are nested in 
$\overrightarrow{V_{i,a}}\overrightarrow{V_{i,a+1}}$.
Without loss of generality, $v \prec p $ in $\overrightarrow{V_{i,a}}$
and $q \prec w$ in $\overrightarrow{V_{i,a+1}}$.
Thus $v \prec p$ and $q\prec w$ in $C_i$. 
If $v \prec p \prec w$ in $C_i$, then $m(p)\geq m(w)$, which contradicts the fact that $g(w) > g(p)$. 
Thus $v \prec w \prec p$.
If $v \prec q \prec w$ in $C_i$, then $vw$ crosses $pq$ in the drawing of $G$.
Thus $q \prec v \prec w \prec p$ in $C_i$.
Since $pq\in E(G)$, we have $m(v) \geq m(q)$, which contradicts the fact that $g(q) > g(v)$. 
Therefore the ordering $\overrightarrow{V_{i,a}}\overrightarrow{V_{i,a+1}}$ defines a $1$-queue layout of 
$G[V_{i,a},V_{i,a+1}]$.
This proves (\ref{Structure-AdjacentBags}).

We now prove (\ref{Structure-InterBag}). That is, for $i,a,b\geq 0$, we show that the ordering $\overrightarrow{V_{i,a}}\overrightarrow{V_{i+1,b}}$ defines a $(6\Delta+1)$-queue layout of $G[V_{i,a},V_{i+1,b}]$. Each edge in $G[V_{i,a},V_{i+1,b}]$ is either direct or hooked. We first show that one queue suffices for direct edges in $G[V_{i,a},V_{i+1,b}]$. Suppose on the contrary that there are two direct edges $vw$ and $xy$ in $G[V_{i,a},V_{i+1,b}]$ with $v \prec x$ in $\overrightarrow{V_{i,a}}$ and $y \prec  w$ in $\overrightarrow{V_{i+1,b}}$. Then $vw$ and $xy$ are drawn between $C_i$ and $C_{i+1}$ with $v \prec w$ in $C_i$ and $x \prec y$ in $C_{i+1}$. Thus $vw$ and $xy$ cross. This contradiction shows that  one queue suffices for direct edges in $G[V_{i,a},V_{i+1,b}]$. 

Now consider a hooked edge $vw$ in $G[V_{i,a},V_{i+1,b}]$ with $v\in V_{i,a}$ and $w\in V_{i+1,b}$. 
Let $vxyzw$ be the corresponding path in $G'$. 
Then $vx$, $xy$ and $wz$ are tree edges in $G'$, and $yz$ is a level edge in $G'$. 
Moreover, $x\in V'_{i+1,c}$ for some $c\geq 0$, 
and $y,z\in V'_{i+2,d}$ for some $d\geq 0$ by \cref{WellLayered}(\ref{WellLayered-NonTreeSpan}).
Since $vx$, $xy$ and $wz$ are tree edges, by \cref{WellLayered}(\ref{WellLayered-TreeSpan}), 
$c-\Delta a =\alpha$ and $d-\Delta c =\beta$ and $d-\Delta b =\gamma$ for some $\alpha,\beta,\gamma\in\{0,1,\dots,2\Delta-1\}$. 
Thus $d = \beta + \Delta c = \gamma + \Delta b$, implying
$\Delta (c-b) = \gamma-\beta \leq 2\Delta-1$ and $c-b \leq 1$. 
Similarly, 
$\Delta (b-c) = \beta-\gamma \leq 2\Delta-1$ and $b-c \leq 1$. 
Thus $c\in\{b-1,b,b+1\}$. 
Assign $vw$ to queue $Q^{\eta}_\gamma$ where $\eta=c-b$. This introduces $6\Delta$ queues. 

We claim that this is a valid queue assignment. 
Suppose on the contrary that there are hooked edges $vw$ and $pq$ in $Q^{\eta}_\gamma$ 
with $v \prec  p$ in $\overrightarrow{V_{i,a}}$ and $q \prec w$ in $\overrightarrow{V_{i+1,b}}$.
Let $vxyzw$ be the path corresponding  to $vw$ in $G'$. 
Let $prstq$ be the path corresponding  to $pq$ in $G'$. 
Then $x,y,z,r,s,t$ are distinct vertices, and $x,r\in V'_{i+1,b+\eta}$ and $y,z,s,t\in V'_{i+2,d}$ where $d=\Delta b + \gamma$. 
By \cref{WellLayered}(\ref{WellLayered-InterBag}) and since $v \prec p$ in  $\overrightarrow{V_{i,a}}$, 
we have $x \prec r$ in $\overrightarrow{V_{i+1,c}}$. 
This in turn implies that $y \prec s$ in $\overrightarrow{V_{i+2,d}}$ by \cref{WellLayered}(\ref{WellLayered-InterBag}).
Similarly, by \cref{WellLayered}(\ref{WellLayered-InterBag}) and since $q \prec w$ in $\overrightarrow{V_{i+1,b}}$, 
we have $t \prec z$ in $\overrightarrow{V_{i+2,d}}$. 
This implies that
$y  \prec  t  \prec  s  \prec  z$ or 
$y  \prec  t  \prec  z  \prec  s$ or 
$y  \prec  s  \prec  t  \prec  z  $ or 
$t  \prec  z  \prec  y  \prec  s$ or  
$t  \prec  y  \prec  s  \prec  z$ or  
$t  \prec  y  \prec  z  \prec  s$ in  $\overrightarrow{V_{i+2,d}}$. 
Thus $yz$ and $st$ either nest or cross in 
$\overrightarrow{V_{i+2,d}}$, which contradicts \cref{WellLayered}(\ref{WellLayered-InBag}).
Hence no two edges in $Q^{\eta}_{\gamma}$ nest.
Therefore
$(Q^\eta_j:\eta\in\{-1,0,1\},j\in\{0,1,\dots,2\Delta-1\})$ is a $6\Delta$-queue layout of the hooked edges in 
$G[V_{i,a},V_{i,a+1}]$ using the ordering $\overrightarrow{V_{i,a}}\overrightarrow{V_{i,a+1}}$. 
Including one queue for the direct edges, we obtain a $(6\Delta+1)$-queue layout of $G[V_{i,a},V_{i,a+1}]$ using the ordering $\overrightarrow{V_{i,a}}\overrightarrow{V_{i,a+1}}$. 
This proves (\ref{Structure-InterBag}). 

Finally we prove (\ref{Structure-InBag}). Consider an edge $vw$ with both end points $v$ and $w$ in $V_{i,a}$ for some $i,a\geq 0$. Then $vw$ is a level edge. Let $vxyw$ be the corresponding path in $G'$. 
Then $xy$ is a level non-tree edge of $G'$ with $v,w\in V'_{i+1}$. 
By \cref{WellLayered}(\ref{WellLayered-NonTreeSpan}), both $x$ and $y$ are in $V'_{i+1,b}$ for some $b\geq 0$.
Since $vx$ and $wy$ are tree edges in $G'$
by \cref{WellLayered}(\ref{WellLayered-TreeSpan}), we have $b-\Delta a \in\{0,1,\dots,2\Delta-1\}$.
Assign $vw$ to queue $Q_{b-\Delta a}$. 
Suppose on the contrary that $v \prec p \prec q \prec w$ for two edges $vw$ and $pq$ in $Q_{b-\Delta a}$. 
Let $vxyw$ be the  path in $G'$ corresponding to $vw$. 
Let $pstq$ be the  path in $G'$ corresponding to $pq$. 
Then $x,y,s,t\in V'_{i+1,b}$. 
Note that $vx$, $wy$, $ps$ and $qt$ are tree edges in $G'$, while  $xy$ and $st$ are level edges in $G'$. 
Since $v\prec p$, we have $x \prec s$ in $\overrightarrow{V_{i+1,b}}$ by \cref{WellLayered}(\ref{WellLayered-InterBag}).
Similarly, since $q \prec q$, we have $t \prec y$  in $\overrightarrow{V_{i+1,b}}$ by \cref{WellLayered}(\ref{WellLayered-InterBag}). Thus
$x  \prec s  \prec t \prec  y$ or $x \prec  t  \prec s  \prec y$ or $x \prec  t  \prec y  \prec s$ or $t \prec  y \prec  x  \prec s$ or $t \prec  x  \prec s  \prec y $ or $t \prec  x \prec  y \prec  s$ in $\overrightarrow{V_{i+1,b}}$
 In each case, $xy$ and $st$ either nest or cross, which contradicts 
\cref{WellLayered}(\ref{WellLayered-InBag}).
Thus no two edges in $Q_{b-\Delta a}$ are nested, and
$(Q_j: j \in\{0,1,\dots,2\Delta-1\}$ is a $2\Delta$-queue layout of $G[V_{i,a}]$ using ordering $\overrightarrow{V_{i,a}}$.
This proves (\ref{Structure-InBag}). 
\end{proof}

We now show that \cref{Structure} leads to a $O(\Delta^2)$-queue layout of an arbitrary planar graph.

%\begin{theorem}
%\label{quadratic}
%Every planar graph $G$ with maximum degree $\Delta$ has queue-number at most $12\Delta^2+16\Delta+3$.
%\end{theorem}

\begin{proof}[Proof of \cref{Planar}] 
Let $\{V_{i,a}: i\in \{0,1,\dots,t\}, a\in\{0,1,\dots,n_i\}$ be the partition of of $V(G)$ from \cref{Structure}. Let $\overrightarrow{V_i}$ be the ordering $\overrightarrow{V_{i,0}}\overrightarrow{V_{i,1}}\dots\overrightarrow{V_{i,n_i}}$ of $V_i$. 
Consider the ordering 
$\overrightarrow{V_0}\overrightarrow{V_1}\dots\overrightarrow{V_t}$ of $V(G)$. 

An edge with both endpoints in $V_{i,a}$ cannot nest an edge with both endpoints in $V_{j,b}$ for $(i,a)\neq(j,b)$, and $2\Delta$ queues suffice for such edges by \cref{Structure}(\ref{Structure-InBag}).
An edge with endpoints in $V_{i,a}$ and $V_{i,a+1}$ cannot nest an edge with endpoints in $V_{j,b}$ and $V_{j,b+1}$ for $(i,a)\neq(j,b)$, and one queue suffice for such edges by \cref{Structure}(\ref{Structure-AdjacentBags}).
By \cref{Structure}(\ref{Structure-LevelEdgeSpan}) this accounts for all level edges.
Thus $2\Delta+1$ queues suffice for level edges. 

For $i,a,b\geq 0$, by \cref{Structure}(\ref{Structure-InterBag}), there is a queue layout  $(Q^{i,a,b}_j:j\in\{1,2,\dots,6\Delta+1\})$  of $G[V_{i,a}, V_{i+1,b}]$. For $j\in\{1,2,\dots,6\Delta+1\}$ and $\alpha\in\{-1,0,\dots,2\Delta\}$, let $$Q^\alpha_j:=\bigcup\{ Q^{i,a,b}_j : i,a,b \geq 0,\, b - \Delta a = \alpha  \}.$$
By \cref{Structure}(\ref{Structure-TreeEdgeSpan}) and (\ref{Structure-NonTreeBinding}), this accounts for all binding edges. 

Suppose that binding edges $vw$ and $pq$ in some $Q^\alpha_j$ are nested with $v \prec p \prec q \prec w$ in our ordering of $V(G)$. Then $v\in V_{i,a}$, $p\in V_{i,b}$, $q\in V_{i+1,c}$ and $w\in V_{i+1,d}$ for some $i,a,b,c,d\geq 0$ with $\alpha=d-\Delta a = c- \Delta b$. Thus $d-c=\Delta(a-b)$. Since $v \prec p \prec q \prec w$ in the ordering, $d-c\geq 0$ and $a-b\leq 0$. Thus $a=b$ and $c=d$, which contradicts \cref{Structure}(\ref{Structure-InterBag}). 

Thus $(2\Delta+2)(6\Delta+1)$ queues suffice for binding edges. 
In total we use $(2\Delta+2)(6\Delta+1)+2\Delta+1= 12\Delta^2+16\Delta+3$ queues
\end{proof}

We emphasise that the vertex ordering used in the proof of \cref{Planar} is identical to that used by \citet{BFGMMRU18}. Our contribution is to show that $O(\Delta^2)$ queues suffice rather than the $O(\Delta^6)$ queues used by \citet{BFGMMRU18}. On the other hand, we now show that up to a constant factor our analysis is tight. That is, the above ordering can produce $\Omega(\Delta^2)$ pairwise nested edges (a so-called `rainbow'), which each must be assigned to a distinct queue. 
%For convenience, we assume that $2\Delta^2$ is a power of 2. 
Start with a rooted binary tree with $2\Delta^2$ leaves. Label the leaves left-right
$$v_{1,1},\dots,v_{1,\Delta}; \dots ; v_{\Delta,1},\dots,v_{\Delta,\Delta}; 
w_{\Delta,\Delta},\dots,w_{\Delta,1}; \dots ; w_{1,\Delta},\dots,w_{1,1}.$$
Subdivide the edge incident to each leaf $v_{i,j}$. 
Let $G$ be the graph obtained by adding the edge $v_{i,j}w_{i,j}$ for $i,j\in\{1,2,\dots,\Delta\}$, as illustrated in \cref{TightExample}. Let $G'$ be the well-layered graph obtained by subdividing the edges of $G$  as described above. 
Thus each edge $v_{i,j}w_{i,j}$ is replaced by a path $v_{i,j}x_{i,j}y_{i,j}z_{i,j}w_{i,j}$. Vertices $y_{i,j}$ and $z_{i,j}$, which are on level 0, are joined by a level edge. Edges $v_{i,j}x_{i,j}$, $x_{i,j}y_{i,j}$ and $z_{i,j}w_{i,j}$ are tree edges. 
The above algorithm does not introduce any parallel edges, since each level edge joins vertices on level 0. Vertices $v_{i,j}$ are on level 1, and vertices $w_{i,j}$ are on level 2. It follows that $g(w_{i,j})=0$ and $g(v_{i,j})=i-1$  for all $i,j$. Thus the vertex ordering of $G$ produced by the above algorithm (after removing subdivision vertices) includes the sequence
$$w_{\Delta,\Delta},\dots,w_{\Delta,1}; \dots ; w_{1,\Delta},\dots,w_{1,1}, 
 v_{1,1},\dots,v_{1,\Delta}; \dots; v_{\Delta,1},\dots,v_{\Delta,\Delta}; .$$
Here, $v_{i,j}w_{i,j}$ is nested with $v_{i',j'}w_{i',j'}$ for $(i,j)\neq (i',j')$. 
Thus $\Delta^2$ queues are needed, as claimed. Curiously this example has maximum degree 3. 

\begin{figure}[!h]
\includegraphics[width=\textwidth]{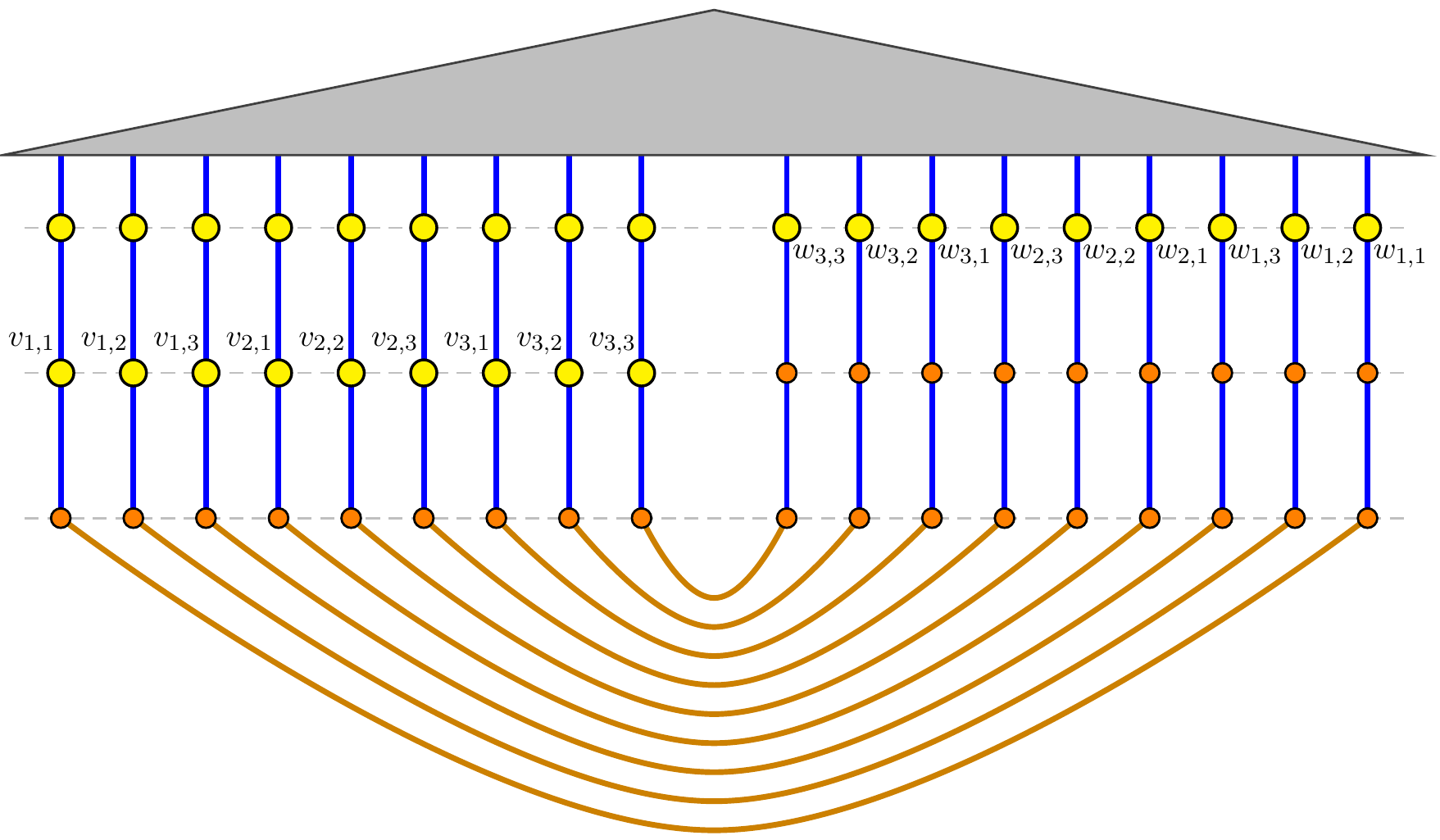}
\caption{Example where the algorithm uses $\Delta^2$ queues. \label{TightExample}}
\end{figure}

\section{Graphs of Bounded Genus}

This section proves our results for graphs of bounded Euler genus (\cref{general} which implies \cref{Genus}). The next lemma is the key. 

%\begin{theorem}[\citep{BFGMMRU18}]
%\label{Bekos}
%For every connected planar graph $G$ with maximum degree $\Delta$, and for every bfs layering $V_0,V_1,\dots, V_t$ of $G$, there is a $(32(2\Delta-1)^6-1)$-queue layout of $G$ with ordering $V_0,V_1,\dots, V_t$.
%\end{theorem}

%\comment{Much of the following result is known (for example, see 4.2.4. in the Mohar--Thomassen textbook), but I have never seen this result with respect to a layering. }

%\begin{theorem}
%Every graph $G$ with Euler genus $g$ and maximum degree $\Delta$ has a $O(g+\Delta^6)$-queue layout. 
%\end{theorem}

\begin{lemma}
\label{MakePlanar}
Let $G$ be a connected graph $G$ with Euler genus $g$. 
For every bfs layering $V_0,V_1,\dots,V_t$ of $G$, 
there is a set $Z\subseteq V(G)$ with at most $2g$ vertices in each layer $V_i$, such that $G-Z$ is planar. 
\end{lemma}

\begin{proof}
Fix an embedding of $G$ in a surface of Euler genus $g$. Say $G$ has $n$ vertices, $m$ edges, and $f$ faces. By Euler's formula, $n-m+f=2-g$. Let $V_0,V_1,\dots,V_t$ be a bfs layering of $G$ rooted at some vertex $r$. Let $T$ be the corresponding bfs spanning tree. Let $D$ be the graph with $V(D)=F(G)$, where for each edge $e$ of $G-E(T)$, if $f_1$ and $f_2$ are the faces of $G$ with $e$ on their boundary, then there is an edge $f_1f_2$ in $D$. (Think of $D$ as the spanning subgraph of $G^*$ consisting of those edges that do not cross edges in $T$.)\  Note that $|V(D)|=f=2-g - n + m$ and $|E(D)|=m-(n-1)= |V(D)|-1+g$. Since $T$ is a tree, $D$ is connected; see \citep[Lemma~11]{DMW17} for a proof. Let $T^*$ be a spanning tree of $D$. Let $Q:=E(D)\setminus E(T^*)$. Thus $|Q|=g$. Say $Q=\{v_1w_1,v_2w_2,\dots,v_gw_g\}$. For $i\in\{1,2,\dots,g\}$, let $Z_i$ be the union of the $v_ir$-path and the $w_ir$-path in $T$, plus the edge $v_iw_i$. Let $Z$ be $Z_1\cup Z_2\cup\dots\cup Z_g$. Say $Z$ has $p$ vertices and $q$ edges. Since $Z$ consists of a subtree of $T$ plus the $g$ edges in $Q$, we have $q = p-1+g$. 

We now describe how to `cut' along the edges of $Z$ to obtain a new graph $G'$; see Figure~\ref{Cutting}. First, each edge $e$ of $Z$ is replaced by two edges $e'$ and $e''$ in $G'$. Each vertex of $G$ incident with no edges in $Z$ is untouched. Consider a vertex $v$ of $G$ incident with edges $e_1,e_2,\dots,e_d$ in $Z$ in clockwise order. In $G'$ replace $v$ by new vertices $v_1,v_2,\dots,v_d$, where $v_i$ is incident with $e'_i$, $e''_{i+1}$ and all the edges incident with $v$ clockwise from $e_i$ to $e_{i+1}$ (exclusive). Here $e_{d+1}$ means $e_1$ and $e''_{d+1}$ means $e''_1$. This operation defines a cyclic ordering of the edges in $G'$ incident with each vertex (where $e''_{i+1}$ is followed by $e'_i$ in the cyclic order at $v_i$). This in turn defines an embedding of $G'$ in some orientable surface. (Note that if $G$ is embedded in a non-orientable surface, then the edge signatures for $G$ are ignored in the embedding of $G'$.)\ 

\begin{figure}[h]
\centering
\includegraphics{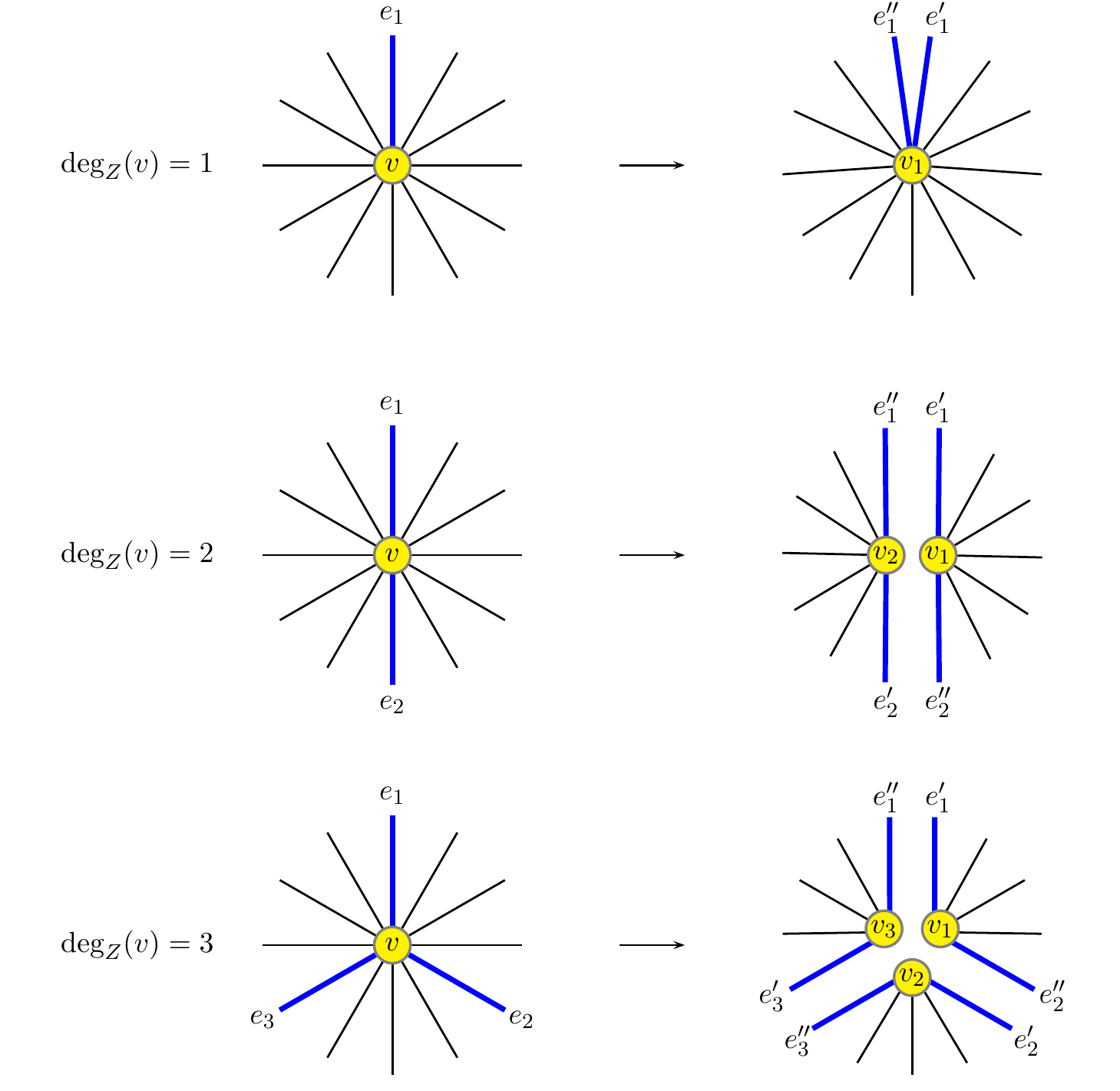}
\caption{Cutting the blue edges in $Z$ at each vertex.\label{Cutting}}
\end{figure}

Say $G'$ has $n'$ vertices and $m'$ edges, and the embedding of $G$' has $f'$ faces and Euler genus $g'$. Each vertex $v$ in $G$ with degree $d$ in $Z$ is replaced by $d$ vertices in $G'$. Each edge in $Z$ is replaced by two edges in $G'$, while each edge of $G-E(Z)$ is maintained in $G'$. Thus 
$$n' = n -p + \sum_{v\in V(G)}\deg_Z(v) = n + 2q -p  = n + 2(p-1+g) - p = n +p - 2 + 2g$$ and 
$m' = m + q = m + p-1 + g$.  Each face of $G$ is preserved in $G'$. Say $r$ new faces are created by the cutting. Thus $f'=f+r$. Since $D$ is connected, it follows that $G'$ is connected. By Euler's formula, $n'-m'+f'=2-g'$. Thus $(n +p - 2 + 2g) - (m + p-1 + g) + (f+r) = 2-g'$, implying $(n-m+f)   - 1 + g + r = 2-g'$. Hence $(2-g)   - 1 + g + r = 2-g'$, implying $g' = 1-r$. Since $r\geq 1$ and $g'\geq 0$, we have $g'=0$ and $r=1$. Therefore $G'$ is planar. 

Note that $G-V(Z)$ is a subgraph of $G'$, and $G-V(Z)$ is planar. By construction, each path $Z_i$ has at most two vertices in each layer $V_j$. Thus $Z$ has at most $2g$ vertices in each $V_j$. 
\end{proof}

%\begin{theorem}
%Every graph $G$ with Euler genus $g$ and maximum degree $\Delta$ has a $O(g+\Delta^6)$-queue layout. 
%\end{theorem}

%\begin{proof}

We need the following lemma of independent interest. 

\begin{lemma}
\label{reorder}
If a graph $G$ has a $k$-queue layout, and $V_0,V_1,\dots,V_t$
is a layering of $G$, then $G$ has a $3k$-queue layout using ordering
$V_0,V_1,\dots,V_t$. 
\end{lemma}

\begin{proof}  Say $E_1,E_2,\dots,E_k$ is the edge-partition and $\preceq$ is the ordering of $V(G)$ in a $k$-queue layout of $G$. For $a\in\{1,2,\dots,k\}$, 
let $X_a$ be the set of edges $vw\in Q_a$ with $v,w\in V_i$ for some $i$; 
let $Y_a$ be the set of edges $vw \in Q_a$ with $v \prec w$ and $v \in V_i$ and $w \in
V_{i+1}$ for some $i$; and 
let $Z_a$ be the set of edges $vw \in Q_a$ with $w \prec v$ and $v \in V_i$ and $w \in
V_{i+1}$ for some $i$.
Then $X_1,Y_1,Z_1,X_2,Y_2,Z_2,\dots,X_k,Y_k,Z_k$ is a partition of $E(G)$. 

Let $\preceq'$ be the ordering $V_0,V_1,\dots,V_t$ of $V(G)$ where each $V_i$ is ordered by $\preceq$. 
No two edges in some set $X_a$ are nested in $\preceq'$, as otherwise the same two edges would be in $Q_a$ and would be nested in $\preceq$. 
Suppose that $v \preceq' p \preceq' q \preceq' w$ for some edges $vw,pq\in Y_a$. 
So $v,p \in V_i$ and $w,q \in V_{i+1}$ for some $i$, and $v \prec p$ and $q \prec w$. 
Now $p \prec q$ by the definition of $Y_a$. 
Hence $v \prec p \prec q \prec w$, which is a contradiction since both $vw$ and $pq$ are in $Q_a$. 
Thus no two edges in $Y_a$ are nested  in $\preceq'$.
By symmetry, no two edges in $Z_a$ are nested in $\preceq'$. 
Hence $\preceq'$ is the ordering in a $3k$-queue layout of $G$.
\end{proof}

We now prove \cref{general}, which says that if $\mathcal{G}$ is a hereditary class of graphs, such that every planar graph in $\mathcal{G}$ has queue-number at most $k$, then every graph in $\mathcal{G}$ with Euler genus $g$ has queue-number at most $3k+4g$.

\begin{proof}[Proof of \cref{general}]
Let $G$ be a graph in $\mathcal{G}$ with Euler genus $g$. 
Since the queue-number of $G$ equals the maximum queue-number of the connected components of $G$, we may assume that $G$ is connected. 
Let $V_0,V_1,\dots,V_t$ be a bfs layering of $G$. 
By \cref{MakePlanar}, there is a set $Z\subseteq V(G)$ with at most $2g$ vertices in each layer $V_i$, such that $G-Z$ is planar. Since $\mathcal{G}$ is hereditary, $G-Z\in \mathcal{G}$, and by assumption $G-Z$ has a $k$-queue layout. 
Note that $V_0 \setminus Z,V_1\setminus Z,\dots,V_t\setminus Z$ is a layering of $G-Z$. 
By \cref{reorder}, $G-Z$ has a $3k$-queue layout using ordering $V_0\setminus Z,V_1 \setminus Z,\dots,V_t \setminus Z$.
Recall that $|V_j \cap Z|\leq 2g$ for all $j\in\{0,1,\dots,t\}$. 
Let $\preceq$ be the ordering
$$V_0 \cap Z, V_0 \setminus Z,\;V_1 \cap Z, V_1\setminus Z, \;\dots ,V_t \cap Z, V_t \setminus Z$$ of $V(G)$. 
where each set $V_j \cap Z$ is ordered arbitrarily, and each set $V_j\setminus Z$ is ordered according to the above $3k$-queue layout of $G-Z$. Edges of $G-Z$ inherit their queue assignment.  We now assign edges incident with vertices in $Z$ to queues. 
For $i \in \{1,\dots,2g\}$ and odd $j\geq 1$, put each edge incident with the $i$-th vertex in $V_j \cap Z$ in a new queue $S_i$. 
For $i \in \{1,\dots,2g\}$ and even $j\geq 0$, put each edge incident with the $i$-th vertex in $V_j \cap Z$ 
(not already assigned to a queue) in a new queue $T_i$. 
Suppose that two edges $vw$ and $pq$ in $S_i$ are nested, where $v\prec p \prec q \prec w$. Say $v\in V_a$ and $p\in V_b$ and $q\in V_c$ and $w\in V_d$. By construction, $a\leq b\leq c\leq d$. Since $vw$ is an edge, $d\leq a+1$. 
At least one endpoint of $vw$ is in $V_j\cap Z$ for some odd $j$, and one endpoint of $pq$ is in $V_\ell\cap Z$ for some odd $\ell$. Since $v,w,p,q$ are distinct, $j\neq \ell$. Thus $|i-j|\geq 2$. This is a contradiction since $a\leq b\leq c\leq d\leq a+1$. Thus $S_i$ is a queue. Similarly $T_i$ is a queue. 
Hence this step introduces $4g$ new queues. 
We obtain a $(3k+4g)$-queue layout of $G$.
\end{proof}

%%%%%%%%%%%%%
\section{Excluded Minors}

Whether the result of  \citet{BFGMMRU18} can be generalised for arbitrary excluded minors is an interesting question. That is, do graphs excluding a fixed minor and with bounded degree have bounded queue-number? It might even be true that graphs excluding a fixed minor have bounded queue-number.

  \let\oldthebibliography=\thebibliography
  \let\endoldthebibliography=\endthebibliography
  \renewenvironment{thebibliography}[1]{%
    \begin{oldthebibliography}{#1}%
      \setlength{\parskip}{0.05ex}%
      \setlength{\itemsep}{0.05ex}%
  }{\end{oldthebibliography}}
%\bibliographystyle{myNatbibStyle}
%\bibliography{myBibliography}

\def\soft#1{\leavevmode\setbox0=\hbox{h}\dimen7=\ht0\advance \dimen7
  by-1ex\relax\if t#1\relax\rlap{\raise.6\dimen7
  \hbox{\kern.3ex\char'47}}#1\relax\else\if T#1\relax
  \rlap{\raise.5\dimen7\hbox{\kern1.3ex\char'47}}#1\relax \else\if
  d#1\relax\rlap{\raise.5\dimen7\hbox{\kern.9ex \char'47}}#1\relax\else\if
  D#1\relax\rlap{\raise.5\dimen7 \hbox{\kern1.4ex\char'47}}#1\relax\else\if
  l#1\relax \rlap{\raise.5\dimen7\hbox{\kern.4ex\char'47}}#1\relax \else\if
  L#1\relax\rlap{\raise.5\dimen7\hbox{\kern.7ex
  \char'47}}#1\relax\else\message{accent \string\soft \space #1 not
  defined!}#1\relax\fi\fi\fi\fi\fi\fi}

\appendix
\section{Unsubdividing}

\citet{DujWoo04} proved that if some $(\leq c)$-subdivision of a graph $G$ has a $k$-queue layout, then $G$ has a $O(k^{2c})$-queue layout.  Here we improve this bound to $O(k^{c+1})$.
 
\begin{lemma}
\label{NewUnsubdivide}
For every $(\leq c)$-subdivision $G'$ of a graph $G$, if $G'$ has a $k$-queue layout using vertex ordering $\preceq$, then $G$ has a $\frac{2k}{2k-1}((2k)^{c+1}-1)$-queue layout using $\preceq$ restricted to $V(G)$. 
\end{lemma}

\begin{proof}
Let $E_1,\dots,E_k$ be the partition of $E(G')$ into queues. For each edge $xy\in E_i$, let $q(xy):=i$. For distinct vertices $a,b\in V(G')$, let $f(a,b):=1$ if $a\prec b$ and let $f(a,b):=-1$ if $b\prec a$. For $\ell\in\{0,1,\dots,c\}$, let $X_\ell$ be the set of edges in $G$ that are subdivided exactly $\ell$ times in $G'$. We will use distinct sets of queues for the $X_\ell$. Consider an edge $vw$ in $X_\ell$ with $v \prec w$. Say $v=x_0,x_1,\dots, x_\ell, x_{\ell+1}=w$ is the corresponding path in $G'$. Let $f(vw):=( f(x_0,x_1), \dots, f( x_\ell, x_{\ell+1}) )$ and $q(vw) := ( q(x_0,x_1),  \dots, q(x_{\ell},x_{\ell+1}) )$. Consider  edges $vw,pq\in X_\ell$ with $v,w,p,q$ distinct and $f(vw)=f(pq)$ and $g(vw)=g(pq)$. Assume $v\prec p$. Say $v=x_0,x_1,\dots, x_\ell, x_{\ell+1}=w$ and $p=y_0,y_1,\dots, x_\ell, x_{\ell+1}=q$ are the paths respectively corresponding to $vw$ and $pq$ in $G'$. Thus $f(x_i,x_{i+1})=f(y_i,y_{i+1})$ and $q(x_ix_{i+1})=q(y_iy_{i+1})$ for $i\in\{0,1,\dots,\ell\}$. Thus $x_ix_{i+1}$ and $y_iy_{i+1}$ are not nested. Since $v=x_0\prec y_0 = p$, it follows by induction that $x_i \prec y_i$ for $i\in\{0,1,\dots,\ell+1\}$. In particular, $w=x_{\ell+1} \prec y_{\ell+1}=q$.  Thus $vw$ and $pq$ are not nested. There are $2^{\ell+1}$ values for $f$, and $k^{\ell+1}$ values for $q$. Thus $(2k)^{\ell+1}$ queues suffice for $X_\ell$. In total, $\sum_{\ell=0}^c(2k)^{\ell+1}=\frac{2k}{2k-1}((2k)^{c+1}-1)$ queues suffice for $G$. 
\end{proof}

\end{document}